\documentclass[11pt]{amsart}

\usepackage{amsmath,amsthm,amsfonts,amssymb,amscd,amsbsy,dsfont,hyperref}

\usepackage{palatino}

\usepackage[margin=3cm]{geometry}

\newcommand{\Ric}{\operatorname{Ric}}

\newcommand{\g}{\mathrm g}

\newcommand{\spn}{\operatorname{span}}

\newcommand{\Id}{\operatorname{Id}}

\newcommand{\rnk}{\operatorname{rank}}
\newcommand{\tr}{\operatorname{trace}}

\newcommand{\R}{\mathbf R}

\renewcommand{\geq}{\geqslant}

\allowdisplaybreaks

\newtheorem{theorem}{Theorem}[]
\newtheorem{lemma}[theorem]{Lemma}
\newtheorem{proposition}[theorem]{Proposition}

\newtheorem{mainthm}{\sc Theorem}

\theoremstyle{definition}
\newtheorem{definition}[theorem]{Definition}

\theoremstyle{remark}
\newtheorem{remark}[theorem]{Remark}
\newtheorem{convention}[theorem]{Convention}

\title{Almost isotropic K\"ahler manifolds.}
\author[B. Schmidt]{Benjamin Schmidt}
\author[K. Shankar]{Krishnan Shankar}
\author[R. Spatzier]{Ralf Spatzier}

\address{\begin{tabular}{lll}
Michigan State University & University of Oklahoma & Univeristy of Michigan  \\
Dept. of Mathematics & Dept. of Mathematics & Dept. of Mathematics \\
619 Red Cedar Road & 601 Elm Avenue & 530 Church Street  \\ 
East Lansing, MI, 48824 & Norman, OK, 73019 & Ann Arbor, MI, 48109\\
{\tt schmidt@math.msu.edu} & {\tt Krishnan.Shankar-1@ou.edu} & {\tt spatzier@umich.edu}
\end{tabular}
}

\allowdisplaybreaks
\numberwithin{equation}{section}
\numberwithin{theorem}{section}


\date{\today}
\begin{document}

\newcommand{\spacing}[1]{\renewcommand{\baselinestretch}{#1}\large\normalsize}
\spacing{1.2}

\thispagestyle{empty}

\begin{abstract}
Let $M$ be a complete Riemannian manifold and suppose $p\in M$. For each unit vector $v \in T_p M$, the \textit{Jacobi operator}, $\mathcal{J}_v: v^\perp \rightarrow v^\perp$ is the symmetric endomorphism, $\mathcal{J}_v(w) = R(w,v)v$.  Then $p$ is an \textit{isotropic point} if there exists a constant $\kappa_p \in \mathbf{R}$ such that $\mathcal{J}_v = \kappa_p \Id_{v^\perp}$ for each unit vector $v \in T_pM$.  If all points are isotropic, then $M$ is said to be isotropic; it is a classical result of Schur that isotropic manifolds of dimension at least 3 have constant sectional curvatures.  

In this paper we consider \textit{almost isotropic manifolds}, i.e. manifolds having the property that for each $p \in M$, there exists a constant $\kappa_p \in \mathbb{R}$, such that the Jacobi operators $\mathcal{J}_v$ satisfy $\text{rank}(\mathcal{J}_v - \kappa_p \Id_{v^\perp}) \leq 1$ for each unit vector $v \in T_pM$. Our main theorem classifies the almost isotropic simply connected K\"ahler manifolds, proving that those of dimension $d=2n \geq 4$ are either isometric to complex projective space or complex hyperbolic space or are totally geodesically foliated by leaves isometric to $\mathbf{C}^{n-1}.$ 
\end{abstract}

\maketitle

\section{Introduction}
Let $M$ be a Riemannian manifold with tangent bundle $TM$ and Levi-Civita connection $\nabla$.  Defining the curvature tensor $R:TM^3 \rightarrow TM$ by $R(X,Y)Z=[\nabla_X,\nabla_Y]Z-\nabla_{[X,Y]} Z$ for sections $X,Y,Z \in \Gamma (TM)$, the \textit{Jacobi operator} of a unit vector $v \in T_pM$ is the symmetric endomorphism $\mathcal{J}_v:v^{\perp} \rightarrow v^{\perp}$ defined by $\mathcal{J}_v(w)=R(w,v)v.$  

The eigenvalues of the Jacobi operator $\mathcal{J}_v$ determine the sectional curvatures of two dimensional subspaces of $T_pM$ containing the vector $v$.  As sectional curvatures determine the curvature tensor, one may hope to classify manifolds with highly restricted Jacobi operators, especially as one removes the dependence of the spectra of $\mathcal{J}_v$ on $v \in TM$.  

For instance, a point $p\in M$ is an \textit{isotropic point} if all two dimensional subspaces of $T_pM$ have equal sectional curvatures. Equivalently, $p \in M$ is isotropic if there exists $\kappa_p \in \mathbf{R}$ such that $\mathcal{J}_v=\kappa_p \Id_{v^{\perp}}$ for each unit vector $v\in T_pM$.  A Riemannian manifold is \textit{isotropic} when each of its points is isotropic.   Schur proved that  connected isotropic manifolds of dimension at least three have constant sectional curvatures \cite{sc}. 

Osserman's Conjecture provides another framework for studying Riemannian manifolds with restricted Jacobi operators. A Riemannian manifold is \textit{Osserman} if the eigenvalues of the Jacobi operators $\mathcal{J}_v$ are independent of the unit vector $v \in TM$.   Riemannian manifolds in which the local isometries act transitively on unit tangent vectors are obviously Ossermann.  These spaces, consisting of the rank one locally symmetric spaces and flat spaces, are conjecturally the only Ossermann manifolds \cite{os}.  This conjecture is known to hold except possibly in dimension sixteen \cite{chi, gi, gsv, ni1, ni2, ni3}.

In this paper, we consider manifolds in which Jacobi operators $\mathcal{J}_v$ may have, a priori, spectra varying with the vector $v$, but share an eigenvalue of large multiplicity.  Theorem \ref{main} below is a K\"ahler analogue of Schur's theorem.  The complete, connected, and simply connected K\"ahler manifolds of constant holomorphic sectional curvatures consist of the manifolds $\mathbf{C}^n$, $\mathbf{C}\mathbf{H}^n$, and $\mathbf{C}\mathbf{P}^n$ equipped with Riemannian symmetric metrics \cite{haw, igu}.  While only points in $\mathbf{C}^n$ are isotropic, points in all of these spaces are \textit{almost isotropic} as now defined. 
\vskip 3pt

\noindent \textbf{Definition.} A point $p \in M$ is \textit{almost isotropic} if there exists $\kappa_p \in \mathbf{R}$ such that $\rnk(\mathcal{J}_v-\kappa_p \Id_{v^{\perp}})\leq 1$ for each unit vector $v \in T_pM$.  A Riemannian manifold is \textit{almost isotropic} if each of its points is almost isotropic.\vskip 3pt

The constant $\kappa_p$ associated to an almost isotropic point is unique.  This is obvious in dimension other than three.  In dimension three, a point cannot have distinct isotropic constants $\kappa_1 \neq \kappa_2$ since otherwise the assignment $v \mapsto\ker(\mathcal{J}_v -\kappa_1 \Id_{v^{\perp}})$ defines a continuous nonvanishing line field on the two dimensional unit sphere.  We now write $\kappa$ for $\kappa_p$ and emphasize the point only when there might be some ambiguity.

Note that if $\Sigma$ is any Riemannian surface, then each point in the Riemannian product $M=\Sigma \times \mathbf{R}^n$ is an almost isotropic point with constant $0$.  The isotropic points in $M$ are those that project to flat points in $\Sigma$.

\begin{mainthm}\label{main}
Let $M$ be a complete, connected and simply connected, almost isotropic K\"ahler manifold of real dimension $d=2n \geq 4$.  Then $\kappa:M \rightarrow \mathbf{R}$ is constant. Moreover,
\begin{enumerate}
\item  if $\kappa>0$, then $M$ is isometric to a metric of constant holomorphic sectional curvature $4\kappa$ on the $n$-dimensional complex projective space $\mathbf{C}\mathbf{P}^n;$ 

\item if $\kappa<0$, then $M$ is isometric to a metric of constant holomorphic sectional curvature $4\kappa$ on the $n$-dimensional complex hyperbolic space $\mathbf{C}\mathbf{H}^n$;

\item if $\kappa=0$, then the set of nonisotropic points $\mathcal{O} \subset M$ admits a $(d-2)$-dimensional parallel tangent distribution that is tangent to a foliation by complete and totally geodesic leaves isometric to $\mathbf{C}^{n-1}$.
\end{enumerate}
\end{mainthm}

In case (3), each point $p \in \mathcal{O}$ admits a neighborhood that is isometric to a metric product of $\mathbf{C}^{n-1}$ with a Riemannian surface.  This local product structure extends to a global product structure when the metric is real analytic \cite[Theorem 8]{abe}.  The authors are unsure whether there is a global product structure in the smooth category (compare with \cite[Corollary 2]{li} and  \cite[Section 4]{zhe}).

It is worth noting that Theorem \ref{main} does not hold in the purely Riemannian setting.  In dimension three, Riemannian manifolds that pointwise have constant vector curvature are almost isotropic; there exist infinite dimensional moduli spaces of such manifolds \cite{ScWo1,ScWo2}.  Almost Hermitian manifolds with pointwise constant antiholomorphic sectional curvatures are also almost isotropic; there exist nontrivial examples of such manifolds \cite{agi,sa}.

The main work in proving Theorem \ref{main} is to first classify the possible curvature tensors at a point in an almost isotropic K\"ahler manifold.  To describe this classification, let $(V,\langle \cdot, \cdot \rangle)$ denote a real inner product space.  An \textit{algebraic curvature tensor} on $V$ is a tensor $R:V^3 \rightarrow V$, $(x,y,z) \mapsto R(x,y)z$, that satisfies the \textit{curvature symmetries}: $$\langle R(x,y)z, w\rangle = -\langle R(y,x)z, w\rangle=\langle R(z,w)x, y \rangle$$ $$R(x,y)z+R(y,z)x+R(z,x)y=0.$$ 

Given a skew symmetric endomorphism $A:V \rightarrow V$, define tensors $R_1:V^3 \rightarrow V$ and $R_A:V^3 \rightarrow V$ by $$R_1(x,y)z=\langle y,z\rangle x-\langle x,z\rangle y$$ $$R_A(x,y)z=2\langle x, Ay \rangle Az+\langle x,Az\rangle Ay-\langle y,Az\rangle Ax.$$  It is straightforward to verify that $R_1$ and $R_A$ are algebraic curvature tensors.

A \textit{complex model} consists of an even dimensional real inner product space $(V,\langle \cdot, \cdot \rangle)$ and an orthogonal almost complex structure $J:V \rightarrow V$.  A \textit{holomorphic plane} is a two dimensional $J$-invariant subspace of $V$.  An algebraic curvature tensor $R$ has \textit{constant holomorphic curvatures} when all holomorphic planes have equal sectional curvature.  An algebraic curvature tensor $R$ on a complex model is \textit{K\"ahler} if it satisfies the following additional \textit{K\"ahler symmetry}: $$R(x,y)z=R(Jx,Jy)z.$$  It may be verified that the algebraic curvature tensor $R=\kappa(R_1+R_J)$ is K\"ahler, almost isotropic with associated constant $\kappa$, and has constant holomorphic curvatures $4\kappa$; the curvature tensor at a point in $\mathbf{C}^n$, $\mathbf{C}\mathbf{P}^n$, or $\mathbf{C}\mathbf{H}^n$ is of this form with $\kappa=0$, $\kappa>0$, and $\kappa<0$, respectively.

Given a subspace $W$ of $V$, let $\pi^W:V \rightarrow W$ denote the orthogonal projection of $V$ onto $W$.

\begin{mainthm}\label{kahleralgebraic}
The algebraic curvature tensors $R$ on a complex model $(V,\langle \cdot, \cdot \rangle, J)$ that are K\"ahler and almost isotropic with associated constant $\kappa$ are classified as follows.

\begin{enumerate}
\item If $\dim(V)=2$, then $R=\kappa R_1=\frac{\kappa}{4}(R_1+R_J).$ \vskip 2pt

\item If $\dim(V)=4$ and $\kappa \neq 0$, then there exists constants $\tau=\pm 1$ and $\mu_1, \mu_2 \in \mathbf{R}$ and orthogonal holomorphic planes $W_1$ and $W_2$ such that $\mu_1 \cdot \mu_2=\kappa/\tau$, and letting $A=J \circ (\mu_1 \pi^{W_1}+\mu_2 \pi^{W_2})$, $R=\kappa R_1+\tau R_{A}$. \vskip 2pt

\item If $\dim(V) \geq 6$ and $\kappa \neq 0$, then $R=\kappa(R_1+R_J)$.\vskip 2pt 

\item If $\dim(V) \geq 4$ and $\kappa=0$, then there exists $c \in \mathbf{R}$ and a holomorphic plane $W$ such that, letting $A=J \circ \pi^W$, $R=c R_A.$
\end{enumerate}

\end{mainthm}

\begin{remark}\label{theoremnote}
Note that in (2), when $\mu_1=\mu_2:=\mu$, the equalities $A=\mu J$ and $R_A=\mu^2 R_J=\kappa/\tau R_J$ hold.  Therefore, if $\mu_1=\mu_2$, then $R=\kappa(R_1+R_J)$ as in (3).  
\end{remark}

Theorem \ref{kahleralgebraic} is proved in Section \ref{Kahler} using the K\"ahler symmetry and the following classification of almost isotropic algebraic curvature tensors obtained in Section \ref{classify}.  

\begin{mainthm}\label{algebraic}
Let $(V \langle \cdot, \cdot \rangle)$ be a real inner product space and let $R:V^3 \rightarrow V$ be an almost isotropic algebraic curvature tensor with associated constant $\kappa$.  Then either $R=\kappa R_1$ is isotropic, or there exists a constant $\tau=\pm1$ and a nonzero skew symmetric endomorphism $A:V \rightarrow V$ such that $R=\kappa R_1+\tau R_A$.  
\end{mainthm}

Almost isotropic algebraic curvature tensors give rise to totally geodesic tangent distributions on unit spheres of codimension at most one as will be explained.  A classification of such distributions is carried out in Section \ref{totgeo}.  The concluding Section \ref{proof examples} consists of the proof of Theorem \ref{main}. 

\section{Codimension one totally geodesic distributions on unit spheres.} \label{totgeo}
Codimension one totally geodesic distributions on unit spheres are classified in this section.  These distributions arise in the study of almost isotropic algebraic curvature tensors undertaken in the next section.  

Subspaces in a distribution, as defined here, may have varying dimensions and may not vary continuously.
\begin{definition}
Let $X$ be a smooth manifold.  

\begin{enumerate}
\item A \textit{distribution} $D$ on $X$ is an assignment to each point $x \in X$ a subspace of $T_xX$.  Such an assignment is denoted by $D:X \ni x \mapsto D_x \subset T_xX.$  \item The \textit{dimension} of $D$, denoted by $\dim(D)$, is the minimum value of the function $x \mapsto \dim(D_x)$; its \textit{codimension} is $\dim(X)-\dim(D)$. 
\item The \textit{singular set} of $D$ is the subset $\mathcal{X}=\{x \in X\, \vert\, \dim(D_x) >\dim(D)\}.$  The distribution $D$ is \textit{nonsingular} when $\mathcal{X}=\emptyset.$
\item The distribution $D$ is \textit{totally geodesic} with respect to a Riemannian metric on $X$ if each geodesic in $X$ is either everywhere or nowhere tangent to $D$.
\end{enumerate}
\end{definition}

Let $(V,\langle \cdot, \cdot \rangle)$ be a $d$-dimensional inner product space and $S$ its unit sphere equipped with the induced Riemannian metric.  For each $s \in S$, parallel translation in $V$ defines an isomorphism between $s^{\perp}$ and $T_sS$.  This isomorphism is used without further mention below.  

Given an endomorphism $A \in \mathcal{L}(V,V)$, let $[A]=\{\lambda A\, \vert\, \lambda \in \mathbf{R},\,\,\, \lambda \neq 0\}$ denote its projective class.  Each skew symmetric projective class $[A]$ gives rise to a totally geodesic distribution $D[A]$ on $S$ of codimension at most one that we now describe.

Fix $A\in [A]$.  As $A$ is skew symmetric, $\langle v, Av \rangle=0$ for each $v \in V$.  Given $s \in S$, let $D[A]_s:=\spn(s,As)^{\perp}=\{w \in V\, \vert\, \langle w,s \rangle=\langle w, As \rangle =0\}.$ The assignment $D[A]: S \ni s \mapsto D[A]_s \subset T_vS$ defines a distribution of codimension at most one on $S$.  The distribution $D[A]$ only depends on the projective class $[A]$ and has codimension one if and only if $[A]\neq [0]$.  If $[A] \neq 0$, then $D[A]$ has singular set $\mathcal{X}=\{s \in S\, \vert \, D_s=T_sS\}=\ker(A) \cap S$.

As $A$ is skew symmetric, the vector field $Z: S \ni s \mapsto Z_s:=As \in T_v S$ is a Killing field.  The following lemma explains why $D[A]$ is totally geodesic.

\begin{lemma}\label{l3}
Let $Z$ be a Killing field on a complete Riemannian manifold $(X,\g)$.  If a geodesic $c:\mathbf{R} \rightarrow X$ satisfies $\g(\dot{c}(0),Z(c(0)))=0$, then $\g(\dot{c}(t),Z(c(t)))=0$ for all $t \in \R$.
\end{lemma}  

\begin{proof}
As $c(t)$ is a geodesic and $Z$ is a Killing field, $\dot{c}\,\g(\dot{c},Z)=\g(\nabla_{\dot{c}} \dot{c}, Z)+\g(\dot{c}, \nabla_{\dot{c}} Z)=0+0=0.$
\end{proof}

The fundamental theorem of projective geometry is used in \cite{halu} to establish that all \textit{nonsingular} codimension one totally geodesic distributions on unit spheres arise in this fashion.  

\begin{theorem}\label{hl}[Hangan and Lutz]
If $D$ is a nonsingular codimension one totally geodesic distribution on the unit sphere $S$, then there exists a projective class $[A]$ of nondegenerate skew symmetric endomorphisms such that $D=D[A]$.  
\end{theorem}

Skew symmetric endomorphisms have even rank.  It follows that $d=\dim(V)$ is even in Theorem \ref{hl}.  Theorem \ref{extend} below extends Theorem \ref{hl} to allow nonempty singular sets $\mathcal{X}$.  We begin with two lemmas. 

\begin{lemma}\label{sym}
Let $D$ be a totally geodesic distribution on $S$.  If $x, y \in S$ and $x \in D_y$, then $y \in D_x$.
\end{lemma}

\begin{proof}
As $x \in D_y$, $\langle x, y \rangle=0$.  Consider the geodesic $c(t)=\cos(t)y+\sin(t)x$.  As $\dot{c}(0)=x \in D_y=D_{c(0)}$, $-y=\dot{c}(\pi/2) \in D_{c(\pi/2)}=D_x$.  The conclusion holds since $D_x$ is a subspace. 
\end{proof}

\begin{lemma}\label{sphere}
Let $D$ be a codimension one totally geodesic distribution on the unit sphere $S$ with singular set $\mathcal{X}$.  Let $K=\spn(\mathcal{X})$ and $M=K^{\perp}$ so that $V=K\oplus M$.  Let $S_K=K \cap S$ and $S_M=M \cap S$. 

\begin{enumerate}
\item If $m \in S_M$, then $K \subset D_m$.
\item If $m \in S_M$, then $D_m\cap T_mS_M$ has codimension one in $T_m S_M$.
\item If $k \in S_K$, then  $M\subset D_k$.
\item The singular set satisfies $\mathcal{X}=S_K$. 
\item Let $k \in S_K$, $m \in S_M$, and $T \in (0,\pi/2)$.  If $s=\cos(T)k+\sin(T)m$, then $$D_s=(k^{\perp}\cap K) \oplus (D_m \cap T_m S_M)\oplus \spn(-\sin(T)k+\cos(T)m).$$
\end{enumerate}
\end{lemma}

\begin{proof}
\noindent \textit{(1).}  Fix $m \in S_M$. Given $x \in \mathcal{X}$, $m \in x^{\perp}$ and so $m \in D_x$. By Lemma \ref{sym}, $x \in D_m$.  Conclude that $K=\spn(\mathcal{X}) \subset D_m$. \vskip 7pt

\noindent \textit{(2).}  If $m \in S_M$, then $T_mS=K \oplus T_m S_M$.  As $m \notin \mathcal{X}$, $D_m$ has codimension one in $T_mS$.  Assertion (1) now implies assertion (2).\vskip 7pt

\noindent \textit{(3).} Fix $k \in S_K$.  Given $m \in S_M$, $k \in D_m$ by (1).  By Lemma \ref{sym}, $m \in D_k$.  Conclude that $M=\spn(S_M)\subset D_k$. \vskip 7pt

\noindent \textit{(4).}  Assertion (4) holds trivially when $\mathcal{X}=\emptyset$, so assume $\mathcal{X} \neq \emptyset.$ If $k \in S_K$, then $T_k S=T_kS_K \oplus M$.  By (3), $M \subset D_k$.  It remains to prove that $T_kS_K \subset D_k$.  Let $\{x_1,\ldots, x_k\}\subset  \mathcal{X}$ be a spanning set for $K$.  For each index $i$, let $v_i=x_i-\langle x_i,k\rangle k \in T_k S_K$ and note that $\spn(v_1,\ldots,v_k)=k^{\perp} \cap K=T_kS_K$.  As each $v_i$ is tangent to a geodesic that joins $k$ to $x_i \in \mathcal{X}$, each $v_i \in D_k$.  The conclusion follows since $D_k$ is a subspace. \vskip 7pt 

\noindent \textit{(5).} As $T \in (0,\pi/2)$, $s \notin K$ and so $\dim(D_s)=d-2$.  The summands in (5) are each orthogonal to $s$ and are pairwise orthogonal.  Therefore, it suffices to prove that each is a subspace of $D_w$ and that their dimensions sum to $d-2$.  

The first summand $k^{\perp} \cap K$ has dimension $\dim(K)-1$.  If $x \in (k^{\perp} \cap S_K)$, then by (4), $x \in \mathcal{X}$.  Therefore, $s \in D_x$ and by Lemma \ref{sym}, $x \in D_s$.  The containment $(k^{\perp}\cap K) \subset D_s$ follows.

The second summand $D_m \cap T_m S_M$ has codimension one in $T_m S_M$ by (2) and so has dimension $\dim(M)-2$.  Let $y$ be a unit vector in $D_m \cap T_m S_M$.  Then $y \in k^{\perp}$, and by (4), $y \in D_k$. By Lemma \ref{sym}, $k \in D_y$.  As $y \in D_m$, Lemma \ref{sym} implies $m \in D_y$.  Conclude that $\spn(k,m) \subset D_y$ whence $s \in D_y$ and $y \in D_s$.  The containment $(D_m \cap T_m S_W) \subset D_s$ follows.

The third summand $\spn(-\sin(T)u+\cos(T)w)$ is one dimensional.  The geodesic $c(t)=\cos(t)k+\sin(t)m$ satisfies $c(0)=k \in \mathcal{X}$.  Therefore $\dot{c}(T) \in D_{c(T)}$ and $\spn(-\sin(T)u+\cos(T)w) \subset D_s$.

The dimensions of the summands add to $(\dim(K)-1)+(\dim(M)-2)+1=\dim(V)-2=d-2$, concluding the proof.  
\end{proof}

\begin{theorem}\label{extend}
If $D$ is a codimension one totally geodesic distribution on the unit sphere $S$, then there exists a projective class $[A]$ of nonzero skew symmetric linear maps such that $D=D[A]$.
\end{theorem}

\begin{proof}
Adopt the notation of Lemma \ref{sphere}.  Define a distribution $\hat{D}$ on $S_M$ by $\hat{D}_m:=D_m \cap T_m S_M$ for each $m \in S_M$.  By Lemma \ref{sphere}-(2), $\hat{D}$ is a non-singular, codimension one distribution on $S_M$.  The distribution $\hat{D}$ is totally geodesic since the distribution $D$ and the sphere $S_M$ are totally geodesic.  By Theorem \ref{hl}, there is a nondegenerate skew symmetric linear map $\hat{A}:M \rightarrow M$ such that $\hat{D}=D[\hat{A}].$ Extend $\hat{A}$ to a nonzero skew symmetric linear map $A:V \rightarrow V$ defined by $A(k+m):=\hat{A}m$ for each $k+m \in K\oplus M=V$. 

The Theorem holds if the equality of subspaces $D_s=\spn(s,As)^{\perp}$ holds for each $s \in S$.  The desired equality holds for $s \in S_K$ by Lemma \ref{sphere}-(4) and the equality $K=\ker A$.  The desired equality holds for $s \in S_M$ since Lemma \ref{sphere}-(1) implies  
$$D_s=K\oplus \hat{D}_s=
K \oplus (\spn(s,\hat{A}s)^{\perp} \cap M)=\spn(s,As)^{\perp}.$$   For the remaining $s \in S$, there exists $k \in K$, $m\in M$, and $T \in (0,\pi/2)$ such that $s=\cos(T)k+\sin(T)m$.  As $\dim(s,As)^{\perp}=d-2=\dim(D_s),$ it suffices to prove that each of the three summands in Lemma \ref{sphere}-(5) are perpendicular to $As$. 

Vectors in $k^{\perp} \cap K$ are perpendicular to $As$ since $M$ is the image of  $A$.  Vectors in $D_m \cap T_mS_M=\hat{D}_m$ are perpendicular to the vector $\hat{A}m$ and therefore perpendicular to $As=\sin(T)\hat{A}m$. Finally, $\langle As,-\sin(T)k+\cos(T)m \rangle=\sin(T)\cos(T)\langle \hat{A}m,m\rangle=0,$ concluding the proof. 
\end{proof}

\section{Almost isotropic algebraic curvature tensors.} \label{classify} 
Almost isotropic algebraic curvature tensors are classified in this section.  As in the previous section,  $(V, \langle \cdot, \cdot \rangle)$ denotes a $d$-dimensional real inner product space with unit sphere $S$.




\begin{lemma}\label{model}
Given a skew symmetric endomorphism $A:V \rightarrow V$ and constants $\kappa, \tau \in \mathbf{R}$, the tensor $R_{\kappa, \tau, A}:=\kappa R_1+\tau R_A$ is an almost isotropic  algebraic curvature tensor with associated constant $\kappa$.  More precisely, if $s \in S$ and $w \in s^{\perp}$, then $$\mathcal{J}_s(w)=\kappa w + 3\tau\langle w, As \rangle As.$$
\end{lemma}
\begin{proof}
Verify using the formulae for $R_1$ and $R_A$ given in the introduction.
\end{proof}

\begin{convention}\label{iso}
The curvature tensor  $R_{\kappa,\tau, A}$ is isotropic if and only if $\tau A=0.$  If $R_{\kappa, \tau, A}$ is not isotropic, then $$R_{\kappa, \tau, A}=R_{\kappa, \frac{\tau}{|\tau|}, \sqrt{|\tau|}A }.$$  In the remainder of this section, we will always normalize so that $\tau \in \{-1, 0, 1\}$, and in addition, so that $\tau=0$ if and only if $A=0$.
\end{convention}

Theorem \ref{algebraic} above asserts that every almost isotropic algebraic curvature tensor is of the form $R_{\kappa, \tau, A}$ for some $\tau \in  \{-1,0,1\}$ and skew symmetric endomorphism $A$.  As preparation for its proof, let $R$ denotes a fixed almost isotropic algebraic curvature tensor on $V$ with associated constant $\kappa$.

\begin{definition}
The \textit{eigenspace distribution} of $R$ is the distribution $E$ on $S$ defined by $E:S \ni s \mapsto \ker(\mathcal{J}_s-\kappa \Id_{s^{\perp}}) \subset T_sS.$ 
 \end{definition}
 
The eigenspace distribution of $R$ is of codimension at most one and is of codimension one if and only if $R$ is \textit{not} isotropic.

\begin{lemma}\label{l1}
For an orthonormal pair $\{v,w\} \subset S$, $\sec(v,w)=\kappa$ if and only if $v \in E_w$.  
\end{lemma}

\begin{proof}
If $v \in E_w$, then $\sec(v,w)=\langle \mathcal{J}_w(v),v \rangle=\langle \kappa v,v\rangle=\kappa.$  Conversely, assume that $\sec(v,w)=\kappa$.  The desired conclusion $v \in E_w$ holds trivially if $E_w=w^{\perp}$, so it may also be assumed that $E_w$ is a proper subspace of $w^{\perp}$. There exists an orthonormal basis $\{e_1, \ldots, e_{d-1}\}$ of $w^{\perp}$ consisting of eigenvectors of $\mathcal{J}_w$ with eigenvalues  $\lambda_i=\kappa$ for each $i\neq d-1$ and $\lambda_{d-1}\neq \kappa$.  Writing $v=\sum_{i=1}^{d-1} \alpha_i e_i$,  $$\kappa=\sec(v,w)=\langle \mathcal{J}_w(v), v \rangle =\kappa(\alpha_1^2+\cdots+\alpha_{d-2}^2)+\lambda_{d-1}\alpha_{d-1}^2=\kappa+(\lambda_{d-1}-\kappa)\alpha_{d-1}^2.$$ Therefore $\alpha_{d-1}=0$, concluding the proof.
\end{proof}

\begin{lemma}\label{l2}
The eigenspace distribution $E$ is totally geodesic. 
\end{lemma}

\begin{proof}
Let $v,w \in S$ with $w\in E_v$. Let $\sigma=\spn(v,w)$.  By Lemma \ref{l1}, $\sec(\sigma)=\kappa$.  If $c(t)=\cos(t)v+\sin(t)w$, then $\sigma=\spn(c(t),\dot{c}(t))$ for all $t \in \mathbf{R}$.  By Lemma \ref{l1}, $\dot{c}(t) \in E_{c(t)}$ for all $t \in \mathbf{R}$.
\end{proof}

By Theorem \ref{extend}, there exists a projective class $[A]$ of skew symmetric linear endomorphisms of $V$ such that $E=D[A]$.  The class $[A]$ is nonzero if and only if $R$ is not an isotropic point.  Let $K=\ker[A]$, $M=K^{\perp}$, $\mathcal{X}=\ker[A] \cap S$ and $S_M=M \cap S$.  Note that $V=K \oplus M$, that $K$ and $M$ are $[A]$-invariant subspaces, and that $S_M$ is odd dimensional.  


\begin{definition}
The \textit{extremal curvature} function is the function $\lambda:S \rightarrow \mathbf{R}$ defined by $\lambda(s)=\tr(\mathcal{J}_s)-(d-2)\kappa$.
\end{definition}

\begin{remark}\label{use}
Note that $\lambda$ is a continuous function on $S$ and that $\lambda=\kappa$ precisely on $\mathcal{X}$. When $R=R_{\kappa,\tau,A}$, Lemma \ref{model} implies that  $$\lambda(s)=\kappa-3\tau\langle s, A^2s\rangle.$$
\end{remark} 

\begin{lemma}\label{jacobi}
For $s \in S$, the Jacobi operator $\mathcal{J}_s:s^{\perp} \rightarrow s^{\perp}$ is given by 
\begin{enumerate}
\item $\mathcal{J}_s(w)=\kappa w$ for $s \in \mathcal{X}$,
\item $\mathcal{J}_s(w)=\kappa w +(\lambda(s)-\kappa)\frac{\langle w, As\rangle}{\langle As, As \rangle}As$ for $s \in S \setminus \mathcal{X}.$
\end{enumerate}
\end{lemma}

\begin{proof}
Assertion (1) is obvious.  As for (2), first observe that the expression for $\mathcal{J}_s$ is independent of a choice of a representative $A \in [A]$.  Fix $A \in [A]$ and note that $As$ is an eigenvector of $\mathcal{J}_s$ with eigenvalue $\lambda(s)$.  Given a vector $w \in s^{\perp}$, let $w_1=w-\frac{\langle w, As \rangle}{\langle As, As \rangle}As$ and $w_2=\frac{\langle w, As\rangle}{\langle As, As \rangle}As$.  The vectors $w_1$ and $w_2$ are eigenvectors of $\mathcal{J}_s$ with corresponding eigenvalues $\kappa$ and $\lambda(s)$.  The desired formula follows since $w=w_1+w_2$. 
\end{proof}

\begin{lemma}\label{equal}
If $v,w \in V$ are orthonormal, then $$(\lambda(v)-\kappa)\langle Aw, Aw \rangle=(\lambda(w)-\kappa) \langle Av, Av \rangle.$$
\end{lemma}

\begin{proof}
Whether or not the equality holds is independent of the choice of representative $A \in [A]$ and moreover, holds trivially if $[A]=0$.  If $[A]\neq 0$ and $\sec(v,w) \neq \kappa$, then the equality follows from Lemma \ref{jacobi}-(2) and the equalities $\langle \mathcal{J}_v(w),w\rangle=\sec(v,w)=\langle \mathcal{J}_w(v),v\rangle$. 

To verify the desired equality holds for all orthonormal pairs, let $$X=\{(v,w) \in S \times S\, \vert\, \langle v, w \rangle =0\}.$$  The function $f:X \rightarrow \mathbf{R}$ defined by $$f((v,w))=(\lambda(v)-\kappa)\langle Aw, Aw \rangle-(\lambda(w)-\kappa) \langle Av, Av \rangle$$ is continuous and, by the previous paragraph, has value zero on the subset of orthonormal pairs spanning planes of sectional curvature other than $\kappa$.  As $[A]\neq 0$, this subset is dense in $X$, whence $f$ vanishes identically. 
\end{proof}

\begin{lemma}\label{lam}
For each $A \in [A]$ there exists a constant $\tau:=\tau(A) \in \mathbf{R}$ such that for each $s \in S$,  
$\lambda(s)=\kappa+ 3\tau\langle As, As\rangle.$
\end{lemma}

\begin{proof}
If $[A]=0$, then the desired equality holds with $\tau=0$.  Now assume that $[A] \neq 0$ and fix $A \in A$.  Define a function $\phi:S \setminus \mathcal{X} \rightarrow \mathbf{R}$ by  $\phi(s)=\frac{\lambda(s)-\kappa}{3\langle As, As \rangle}$.  It suffices to prove that $\phi$ is a constant function.  By Lemma \ref{equal}, $\phi$ has equal value on pairs of orthogonal vectors in $S \setminus \mathcal{X}$.  

We first claim that the restriction of $\phi$ to $S_M$ is constant.  If $\dim(M)=2$, then vectors in $S_M$ are eigenvectors of $A^2$ with a common eigenvalue from which the claim follows.  If $\dim(M)>2$, then  recall that $M$ is even dimensional and $\dim(S_M) \geq 3$.  Therefore, if $m_1, m_2 \in S_M$, there exists $m \in S_M$ orthogonal to both of $m_1$ and $m_2$.  Therefore, $\phi(m_1)=\phi(m)=\phi(m_2)$, concluding the proof of this claim.  Let $\tau:=\phi(S_M)$.

Now if, $s \in S \setminus \mathcal{X}$, there exists $\bar{s} \in s^{\perp} \cap S_M$.  Therefore $\phi(s)=\phi(\bar{s})=\tau$, concluding the proof.



\end{proof}

\noindent \textit{Proof of Theorem \ref{algebraic}.}  
Let $R$ be a $\kappa$-almost isotropic curvature tensor on $(V, \langle \cdot, \cdot \rangle)$.  In the notation of Lemma \ref{lam}, there exists $A \in [A]$ such that $\tau=\tau(A) \in \{-1,0,1\}$.  The algebraic curvature tensors $R$ and $R_{\kappa,\tau,A}$ have equal Jacobi operators by Lemmas \ref{model}, \ref{jacobi}, and \ref{lam}.  Therefore, these algebraic curvature tensors have equal sectional curvatures.  The Theorem follows since an algebraic curvature tensor is determined by its sectional curvatures.\qed\\

If $R$ is an algebraic curvature tensor on $(V,\langle \cdot, \cdot \rangle)$, then the associated Ricci tensor is the symmetric tensor $\Ric:V \times V \rightarrow \mathbf{R}$ defined by $\Ric(v,w)=\tr(x \mapsto R(x,v)w)$.  The algebraic curvature tensor $R$ is \textit{Einstein} if there exists a constant $c\in \mathbf{R}$ such that $\Ric(\cdot,\cdot)=c\,\langle \cdot , \cdot \rangle$.   

\begin{lemma}\label{einstein}
Let $R=R_{\kappa, \tau, A}$ be an almost isotropic curvature tensor.  Then for each $v \in S$, $$ \Ric(v,v)=(d-1)\kappa -3\tau\langle v,A^2v\rangle.$$  In particular, $R$ is Einstein if and only if $A^2$ is a multiple of the identity.
\end{lemma}

\begin{proof}
The first assertion implies the second assertion. If $v \in S$, then $\Ric(v,v)=\tr\mathcal{J}_v=\lambda(v)+(d-2)\kappa$.  By Lemma \ref{lam}, $\Ric(v,v)=(d-1)\kappa -3\tau\langle v,A^2v\rangle$.  
\end{proof}

\section{K\"ahler Almost isotropic algebraic curvature tensors.}  \label{Kahler}

K\"ahler almost isotropic curvature tensors are classified in this section.  Let $R$ denote a K\"ahler $\kappa$-almost isotropic curvature tensor on a complex model $(V,\langle, \cdot, \cdot \rangle)$.  By Theorem \ref{algebraic} there exists $\tau \in \{-1,0,1\}$ and a skew symmetric endomorphism $A:V \rightarrow V$ such that 
\begin{equation}\label{formula}
R(x,y)z=\kappa[\langle y,z\rangle x -\langle x,z\rangle y]+\tau[2\langle x, Ay \rangle Az + \langle x, Az\rangle Ay -\langle y,Az\rangle Ax] \end{equation} for each $x,y,z \in V$. Moreover, $\tau=0$ if and only if $A=0$. 

In the remainder of this section we let $B=AJ$.  Note that the adjoint of $B$ equals $JA$.  Recall that $J$ is an orthogonal almost complex structure for our complex model, $(V,\langle, \cdot, \cdot \rangle)$.

\begin{lemma}\label{calculate}
If $\{x,y\}$ are an orthonormal pair of vectors in $V$, then \begin{eqnarray}\label{one} \kappa[\langle x,Jy \rangle Jy-x]=\\ \nonumber\tau[\langle x, (3A+2JB)y \rangle Ay + \langle x, By \rangle By-\langle y, By \rangle Bx] \end{eqnarray} and \begin{eqnarray}\label{two}  \kappa[\langle y, By \rangle Jx-\langle y, Bx \rangle Jy-\langle x, Ay \rangle y]=\\ \nonumber \tau[2\langle x,(A+JB)y\rangle A^2y+\langle y, A^2y \rangle Ax-\langle x, A^2y\rangle Ay+\langle By,Ay\rangle Bx-\langle Bx,Ay\rangle By].\end{eqnarray}
\end{lemma}

\begin{proof}
To derive (\ref{one}), use (\ref{formula}) to expand the symmetry $R(x,y)y=R(Jx,Jy)y$ and then simplify. To derive (\ref{two}) use (\ref{formula}) to expand the symmetry $R(x,y)Ay=R(Jx,Jy)Ay$ and then simplify.
\end{proof}

\begin{lemma}\label{zero}
If $\dim(V) \geq 4$ and if $\tau=0$, then $\kappa=0$ and $R=0$.
\end{lemma}

\begin{proof}
Fix a unit vector $x \in V$.  As $\dim(V) \geq 4$, there exists a unit vector $y \in \spn(x,Jx)^{\perp}$.  Use (\ref{one}) to conclude $\kappa=0$.
\end{proof}




\begin{lemma}\label{commute?}
If $\tau \neq 0$, then $AJ=JA$ or $AJ=-JA$.
\end{lemma}

\begin{proof}
Let $x,y \in V$ be an orthonormal pair of vectors.  Use Lemma \ref{model} or (\ref{formula}) to calculate 
$$\kappa+3\tau\langle x, Ay \rangle^2=\langle R(x,y)y,x\rangle = \langle R(Jx, Jy)Jy, Jx  \rangle=\kappa+3 \tau \langle Jx, AJy \rangle^2.$$  Therefore, $$\langle x, Ay \rangle^2= \langle x, JAJy \rangle ^2.$$  Equivalently, $$\langle x, (A+JAJ)y\rangle \langle x, (A-JAJ)y\rangle=0.$$ The last equality remains true if $y$ is replaced with a scalar multiple of $y$.  Since $A+JAJ$ and $A-JAJ$ are skew symmetric, $$V= \ker(JAJ-A) \cup \ker(JAJ+A),$$ implying the Lemma.
\end{proof}

\begin{lemma}\label{skewcommute}
If $\tau \neq 0$ and if $AJ=-JA$, then for each $x \in V$, $$\kappa \langle Ax, Ax \rangle=\tau\langle Ax, Ax \rangle^2.$$ In particular, $\kappa \neq 0$.  \end{lemma}

\begin{proof}
The first assertion implies the second since $A \neq 0$.  As $AJ=-JA$, the endomorphism $JA$ is skew symmetric. Therefore, $\langle JAx, x \rangle=0$ and $\langle JAx,Ax \rangle=0$.  Hence, if $x \in S$, then $JAx \in E_x=\spn(x,Ax)^{\perp}$ and $$\kappa \langle JAx, JAx \rangle=\langle R(JAx, x)x, JAx\rangle.$$  As $J$ is orthogonal and $R$ satisfies $R(\cdot,\cdot)\cdot=R(J \cdot, J \cdot)\cdot$, $$\kappa \langle Ax, Ax \rangle=-\langle R(Ax, Jx)x, JAx\rangle.$$  Using (\ref{formula}), the last equality simplifies to $$\kappa \langle Ax, Ax \rangle=\tau\langle Ax, Ax \rangle^2.$$  
\end{proof}

\begin{lemma}\label{commute}
If $\tau \neq 0$, then $AJ=JA$.  In particular, $B=AJ$ is a symmetric endomorphism.
\end{lemma}

\begin{proof}
The first assertion implies the second.  As for the first, we now argue by contradiction.  By Lemma \ref{commute?}, $AJ=-JA$. In particular, $JA$ is skew symmetric.

As $A$ is skew symmetric and nonzero, there is an $A$-invariant two dimensional subspace $\sigma$ on which $A$ is nonzero.  If $u \in \sigma$, then $$\langle Ju,Au \rangle=-\langle u, JAu\rangle=0.$$ Therefore $J(u)$ is orthogonal to $\sigma=\spn(u,Au)$.  Hence $J(\sigma)$ is a two dimensional subspace perpendicular to $\sigma$ on which $A$ is also nonzero.  Let $W=\sigma \oplus J(\sigma)$.  By Lemma \ref{skewcommute}, the restriction of $A^2$ to $W$ equals $-\kappa/\tau \Id_W$.  By Remark \ref{use}, the sectional curvatures of $R$ on the subspace $W$ lie between $\kappa$ and $\lambda:=4\kappa$.

Let $\{e_1,e_2\}$ be an orthonormal basis of $\sigma$.  The above discussion implies that $\{e_1,e_2, e_3:=Je_1, e_4:=Je_2\}$ is an orthonormal $4$-frame in $W$.  By Berger's mixed curvature inequality \cite{ber2,ka},  $$| R(e_1,e_2,e_3,e_4) | \leq | \frac{2}{3}(\lambda-\kappa) |=2|\kappa |.$$ This contradicts $$R(e_1,e_2,e_3,e_4)=R(e_1,e_2,Je_1,Je_2)=R(e_1,e_2,e_1,e_2)=-\sec(\sigma)=-\lambda=-4\kappa$$ and $\kappa \neq0$.
\end{proof}




\begin{lemma}\label{relations}
If $\tau\neq 0$, and if $\{e_1,e_2\}$ is an orthonormal pair of eigenvectors of $B$ with corresponding eigenvalues $\mu_1$ and $\mu_2$, then \begin{equation}\label{three} \kappa(1-\langle e_1,Je_2\rangle^2)=\tau(\mu_1 \mu_2-\mu_2^2\langle e_1,Je_2\rangle^2),\end{equation} and  \begin{equation}\label{four} \kappa\mu_2(1-\langle e_1, Je_2\rangle^2)=\tau \mu_1 \mu_2^2(1-\langle e_1, Je_2\rangle^2).\end{equation}

\end{lemma}

\begin{proof}
As $\tau \neq 0$, Lemma \ref{commute} shows that $B=AJ=JA$ is symmetric.  In particular, $JB=-A$ and  $$ Ae_1=-\mu_1Je_1\,\,\,\,\,\,\,\,\,\, Ae_2=-\mu_2 Je_2.$$ Substitute $x=e_1$ and $y=e_2$ in (\ref{one}) and take the inner product of the resulting vector equation with $e_1$ to derive equality (\ref{three}).  Substitute $x=e_1$ and $y=e_2$ in (\ref{two}) and take the inner product of the resulting vector equation with $Je_1$ to derive equality (\ref{four}).  
\end{proof}

\begin{lemma}\label{eigenspace}
If $\tau \neq 0$, then the eigenspaces of $B$ are $J$-invariant.
\end{lemma}

\begin{proof}
By Lemma \ref{commute}, $B=AJ=JA$.  If $v \in V$ and $\mu \in \mathbf{R}$ satisfy $Bv=\mu v$, then $BJv=AJ^2v=-Av=J^2Av=JBv=\mu Jv.$ 
\end{proof}

\begin{lemma}\label{product}
If $\tau\cdot \kappa \neq 0$ and if $\{e_1,e_2\}$ are an orthonormal pair of eigenvectors of $B$ that do not span a holomorphic plane, then the corresponding eigenvalues $\mu_1$ and $\mu_2$ satisfy $\mu_1 \cdot \mu_2=\kappa/\tau$.
\end{lemma}

\begin{proof}
The hypothesis implies that $\langle e_1, Je_2 \rangle^2 \neq 1$.  By (\ref{three}), $\mu_2 \neq 0$.  By (\ref{four}), the desired equality holds.



\end{proof}

\begin{lemma}\label{special}
If $\tau \cdot \kappa \neq 0$ and if $\dim(V)=4$, then there exists $\mu_1,\mu_2 \in \mathbf{R}$ and orthogonal holomorphic planes $W_1$ and $W_2$ such that $\mu_1\cdot \mu_2=\kappa/\tau$ and $$A=J \circ(\mu_1 \pi^{W_1}+\mu_2 \pi^{W_2}).$$
\end{lemma}

\begin{proof}
Let $e_1$ be a unit eigenvector of $B$ with eigenvalue $-\mu_1$ and set $W_1=\spn(e_1,Je_1)$ and $W_2=W_1^{\perp}$.  Then $W_1$ and $W_2$ are orthogonal holomorphic planes.  They consist of eigenvectors of $B$ by Lemma \ref{eigenspace}.  Let $-\mu_2$ be the eigenvalue corresponding to the subspace $W_2$.  By Lemma \ref{product},  $\mu_1\cdot \mu_2=\kappa/\tau$.  If $v \in W_i$, then $$Av=-J^2Av=-JBv=\mu_iJv,$$ from which the conclusion follows.
\end{proof}

\begin{lemma}\label{Equal}
If $\tau\cdot \kappa \neq 0$, and if $\dim(V) \geq 6$, then $B$ is a nonzero multiple of the identity.
\end{lemma}

\begin{proof}
Lemmas \ref{eigenspace} and \ref{product} imply that $B$ has at most two distinct eigenvalues. Seeking a contradiction, suppose that $B$ has two distinct eigenvalues $\mu_1$ and $\mu_2$ and let $E_1$ and $E_2$ denote the corresponding eigenspaces.  Without loss of generality, $\dim(E_1) \geq 4$.  Applying Lemma \ref{product} to an orthonormal pair in $E_1$ that do not span a holomorphic plane implies that $\mu_1^2=\kappa /\tau$.  On the other hand, by Lemma \ref{product}, $\mu_1 \cdot \mu_2=\kappa / \tau$, a contradiction.

\end{proof}

\begin{lemma}\label{setup}
If $\tau \neq 0$ and $\kappa=0$, then there is a nonzero constant $\mu$ and a holomorphic plane $W$ such that $$A= J \circ \mu\pi^W.$$
\end{lemma}

\begin{proof}
As $\tau \neq 0$, $A \neq 0$ and $B=AJ=JA\neq 0$.  Let $-\mu$ be a nonzero eigenvalue of $B$ and $v\in V$ an eigenvector of eigenvalue $\mu$.  Set $W=\spn(v,Jv)$.  The subspace $W$ consists of eigenvectors with eigenvalue $-\mu$ by Lemma \ref{eigenspace}.  Moreover, if $w \in W$, then $Aw=-J^2Aw=-JBw=\mu Jw$.  If $v \in W^{\perp}$ is an eigenvector of $B$, then its eigenvalue is zero by (\ref{four}).  Therefore $W^{\perp}=\ker(B)=\ker(A),$ concluding the proof.


\end{proof}

\noindent \textit{Proof of Theorem \ref{kahleralgebraic}.\\}  
The curvature tensors appearing in the statement of the Theorem are K\"ahler and almost isotropic with constant $\kappa$ as can be verified directly from their formulae.  These are the only such algebraic curvature tensors as we now argue in each case:\\

\noindent (1):  Obvious.\\

\noindent (2): Immediate from Lemma \ref{special} and the expression for $R$ from Theorem \ref{algebraic}.\\

\noindent (3): By Lemma \ref{zero}, $\kappa \cdot \tau\neq 0$.  By Lemma \ref{Equal} there is a constant $\mu$ such that $B=AJ=JA=\mu \Id$.  Multiplying through by $J$, $A=-\mu J$.  The desired formula is derived by substituting this expression for $A$ into the expression for $R$ from Theorem \ref{algebraic} and then simplifying using the equality $\mu^2=\kappa/\tau$ from Lemma \ref{product}.\\

\noindent (4): If $R$ is isotropic, then the Lemma holds with $c=0$.  If $R$ is not isotropic, then by Lemma \ref{setup}, there exists a holomorphic plane $W$ and a nonzero constant $\mu$ such that $A=\mu(J \circ \pi_W)$.  The conclusion holds with $c=\tau \mu^2$ by (\ref{formula}).

\qed

\begin{definition}
The \textit{nullity space} of an algebraic curvature tensor $R$ on an inner product space $(V\langle \cdot, \cdot \rangle)$ is the subspace $$N=\{v \in V\, \vert\, R(w,v)=0\,\,\text{for every}\,\, w\in V\}.$$  Its dimension is the \textit{index of nullity}.
\end{definition}

\begin{lemma}\label{nullity}
Let $R$ be a Kahler almost isotropic algebraic curvature tensor as in Theorem \ref{kahleralgebraic}-(4) with $c\neq0$.  Then $W^{\perp}=\ker(A)$ is the nullity space of $R$.  In particular, the index of nullity is $\dim(V)-2$.
\end{lemma}

\begin{proof}
The curvature tensor is given by $$R(x,y)z=c[2\langle x, Ay \rangle Az + \langle x, Az\rangle Ay +\langle Ay,z\rangle Ax]$$ where $A=J \circ \pi^W$. The above formula implies $W^{\perp}=\ker(A) \subset N$.  If $w \in W$, then $R(Jw,w)w=3c\langle w, w \rangle Jw$.  Therefore $N \cap W=\{0\}$ from which the Lemma follows.
\end{proof}






\section{Proof of Theorem \ref{main}.} \label{proof examples}
Let $(M,\g)$ denote a complete, connected and simply connected K\"ahler almost isotropic manifold of dimension at least four.  There is a function $\kappa:M \rightarrow \mathbf{R}$ such that each point $p\in M$ is an almost isotropic point with associated constant $\kappa(p)$.   

The proof of Theorem \ref{main} is presented at the end of this section after a few preliminary results.  The next Lemma is most likely well known; we present a proof for the convenience of the reader.

\begin{lemma}\label{split}
Let $B$ be an open connected subset of a Riemannian manifold $(M,\g)$ admitting a pair of transverse, orthogonal, and totally geodesic foliations $\mathcal{F}_1$ and $\mathcal{F}_2$.  Then $B$ is locally isometric to the product $\mathcal{F}_1 \times \mathcal{F}_2$.
\end{lemma}

\begin{proof}
If $H=T\mathcal{F}_1$ and $V=T\mathcal{F}_2$, then the tangent bundle splits orthogonally $TB=H\oplus V$. By de Rham's splitting theorem, it suffices to prove that the distribution $H$ is parallel on $B$.  Let $X$ be a vector field tangent to $H$.  We must show that if $Y$ is a vector field, then $\nabla_{Y} X$ is also tangent to $H$.  

Decompose $Y=Y_1+Y_2$ with $Y_1$ tangent to $H$ and $Y_2$ tangent to $V$.  Then $\nabla_{Y_1} X$ is tangent to $H$ since $H$ is integrable and totally geodesic.  If $Z$ is a vector field tangent to $V$, then since $\g(X,Z)=0$, $$\g(\nabla_{Y_2} X, Z)=-\g(X, \nabla_{Y_2} Z)=0$$ where the last equality holds since $V$ is integrable and totally geodesic.  Therefore $\nabla_{Y_2} X$ is tangent to $H$, concluding the proof.
\end{proof}

\begin{proposition}\label{k12}
If $\dim(M) =4$ and $\kappa(p) \neq 0$, then the constants $\mu_1(p)$ and $\mu_2(p)$ associated to the curvature tensor $R_p$ by Theorem \ref{kahleralgebraic}-(2) are equal.
\end{proposition}

\begin{proof}
If not, then by smoothness of the curvature tensor on $M$, there is a ball $B$ in $M$ about $p$ such that $\kappa(b) \neq 0$ and $\mu_1(b)\neq \mu_2(b)$ for each $b \in B$.  Let $W_1$ and $W_2$ denote the corresponding smooth tangent plane distributions on $B$.  There exist unit vector fields $v_1$ and $v_2$ on $B$ tangent to $W_1$ and $W_2$ respectively.  Letting $\bar{v}_i=J v_i$, we obtain an orthonormal framing $\{v_1,\bar{v}_1,v_2,\bar{v}_2\}$ of $TB$ with $W_i=\spn\{v_i,\bar{v}_i\}$ for $i=1,2$.  Note that $$\sec(W_1)-\kappa=3\tau\mu_1^2$$ $$\sec(W_2)-\kappa=3\tau\mu_2^2.$$ 

The goal of the following calculations is to show that the orthogonal distributions $W_1$ and $W_2$ are integrable and totally geodesic.  As $J$ is parallel, 
\begin{equation}\label{parallel}
\g(\nabla_X JY, Z)=\g(J\nabla_X Y,Z)=-\g(\nabla_X Y, JZ)
\end{equation} for all smooth vector fields $X,Y,Z$.  Use (\ref{parallel}) to conclude \begin{equation}\label{alpha}
\g(\nabla_{v_2} v_2, \bar{v}_1)=-\g(\nabla_{v_2} \bar{v}_2, v_1).\end{equation}

Use the differential Bianchi identity, $$0=(\nabla_{v_2}R)(v_1,\bar{v}_1,v_1,v_2)+(\nabla_{v_1}R)(\bar{v}_1,v_2,v_1,v_2)+(\nabla_{\bar{v}_1}R)(v_2,v_1,v_1,v_2)$$ to derive \begin{equation}\label{beta}
3\tau \mu_1^2\g(\nabla_{v_2} v_2,\bar{v}_1) - 3\kappa\g(\nabla_{v_2} \bar{v}_2,v_1)=0.
\end{equation}  Use (\ref{alpha}), (\ref{beta}), and $\kappa\neq \tau \mu_1^2$ (since $\mu_1 \neq \mu_2$) to conclude \begin{equation}\label{gamma}
\g(\nabla_{v_2} v_2,\bar{v}_1)=\g(\nabla_{v_2} \bar{v}_2,v_1)=0.
\end{equation} Set $w_1:=\bar{v}_1$ and $\bar{w}_1:=Jw_1=-v_1$.  Repeating the above calculations with $w_1$ and $\bar{w}_1$ in place of $v_1$ and $\bar{v}_1$, respectively, yields the following analogue of (\ref{gamma}), \begin{equation}\label{delta}
\g(\nabla_{v_2} v_2,\bar{w}_1)=\g(\nabla_{v_2} \bar{v}_2,w_1)=0,
\end{equation} or equivalently, \begin{equation}\label{epsilon}
\g(\nabla_{v_2} v_2,v_1)=\g(\nabla_{v_2} \bar{v}_2,\bar{v}_1)=0.
\end{equation} Set $w_2:=\bar{v}_2$ and $\bar{w}_2:=Jw_2=-v_2$.  Repeating the above calculations with $w_2$ and $\bar{w}_2$ in place of $v_2$ and $\bar{v}_2$, respectively, yields the following analogues of (\ref{gamma}), (\ref{epsilon}) \begin{equation}
\g(\nabla_{w_2} w_2,\bar{v}_1)=\g(\nabla_{w_2} \bar{w}_2,v_1)=0,
\end{equation} and \begin{equation}
\g(\nabla_{w_2} w_2,v_1)=\g(\nabla_{w_2} \bar{w}_2,\bar{v}_1)=0,
\end{equation} or equivalently, \begin{equation}\label{f}
\g(\nabla_{\bar{v}_2} \bar{v}_2,\bar{v}_1)=\g(\nabla_{\bar{v}_2} v_2,v_1)=0,
\end{equation} and \begin{equation}\label{g}
\g(\nabla_{\bar{v}_2} \bar{v}_2,v_1)=\g(\nabla_{\bar{v}_2} v_2,\bar{v}_1)=0.
\end{equation}

The two dimensional distribution $W_2$ is integrable and totally geodesic by (\ref{gamma}), (\ref{epsilon}), (\ref{f}), and (\ref{g}).  Switching the roles of the indices $1$ and $2$ in the differential Bianchi calculation above, yields the following analogue of (\ref{beta}) \begin{equation}
3\tau \mu_2^2\g(\nabla_{v_1} v_1,\bar{v}_2) - 3\kappa\g(\nabla_{v_1} \bar{v}_1,v_2)=0.
\end{equation}  Now, arguing as in the case of the two dimensional distribution $W_2$, the two dimensional distribution $W_1$ is also integrable and totally geodesic.

 As the tangent plane fields $W_1$ and $W_2$ are orthogonal, integrable, and totally geodesic, $B$ is locally isometric to a Riemannian product by Lemma \ref{split}.   Consequently, $W_2 \subset \ker(\mathcal{J}_{v_1})$.  This is a contradiction since $\kappa \neq 0$ implies that $\dim(\ker{J}_{v_1})\leq 1$.
\end{proof}

\begin{lemma}\label{Einstein}
If $p \in M$ satisfies $\kappa(p)\neq 0$, then $R_p\cong \kappa(R_1+R_J)$.  In particular, $p$ is an Einstein point with Einstein constant $(d+2)\kappa(p)$ and has constant holomorphic curvatures $4\kappa(p)$.
\end{lemma}

\begin{proof}
By Theorem \ref{kahleralgebraic}-(2,3), Remark \ref{theoremnote}, and Proposition \ref{k12}, $R_p\cong \kappa(R_1+R_J)$.  Also, by Lemma \ref{einstein}, we obtain that $p$ has Einstein constant $(d+2)\kappa(p)$.
\end{proof}

\noindent \textit{Proof of Theorem \ref{main}.\\} 
By Lemma \ref{Einstein}, each point in the set $\kappa \neq 0$ is an Einstein point.  By Schur's Theorem for Einstein points, $\kappa$ is constant in each connected component of $\kappa \neq 0$.  By continuity, $\kappa$ is constant.  If $\kappa$ is a nonzero constant, then $M$ has constant holomorphic curvatures $4\kappa$ by Lemma \ref{Einstein}, concluding the proof in this case by \cite{haw,igu}.

Now assume that $\kappa$ is identically zero and let $\mathcal{O}$ denote the subset of nonisotropic points in $M$.  At each point $p \in \mathcal{O}$, there is a holomorphic plane $W_p$ in $T_pM$ given by Theorem \ref{kahleralgebraic}-(4).  By Lemma \ref{nullity}, the distribution $p \mapsto W_p^{\perp}$ is the nullity distribution of the curvature tensor and has constant index of nullity $\dim(M)-2$.  The argument given in \cite[Theorem 8]{abe} now applies verbatim to prove (3).
\qed


\end{document}